\documentclass[article,12pt]{amsart}

\usepackage
[
	a4paper,
	left=1in,
	right=1in,
	top=1in,
	bottom=1in,
]
{geometry}

\usepackage{setspace}

\usepackage{euscript}
\usepackage{amsfonts}
\usepackage{amssymb}
\usepackage{amsmath}
\usepackage{amscd}
\usepackage{amsthm}
\usepackage{enumerate}
\usepackage{hyperref}
\usepackage{fullpage}
\usepackage{algorithm}
\usepackage{algpseudocode}
\usepackage{float}
\usepackage{graphicx}
\usepackage{xcolor}
\usepackage{comment}
\usepackage[shortlabels]{enumitem}
\usepackage{appendix}

\usepackage{amsfonts}
\usepackage{textcomp}
\usepackage{amssymb}
\usepackage{amsmath}
\usepackage{amsthm}
\usepackage{color}
\usepackage{ulem}
\usepackage{fullpage}
\usepackage{enumerate}
\usepackage{enumitem}
\usepackage{mathtools}
\usepackage{interval}

\usepackage{amsthm}
\usepackage{enumerate}
\usepackage{mathtools}
\usepackage{amsmath}
\usepackage{amssymb}
\usepackage{amsfonts}
\usepackage{interval}

\theoremstyle{plain}
\newtheorem{thm}{Theorem}[section]

\newtheorem{prop}[thm]{Proposition}

\numberwithin{equation}{section}

\title[Strong existence and uniqueness]{Strong existence and uniqueness of a calibrated local stochastic volatility model}

\author{Scander Mustapha}
\address{Program in Applied \& Computational Mathematics, Princeton University, Princeton, NJ 08544.}
\email{mustapha@princeton.edu}

\date{May 2024}

\begin{document}
\begin{abstract}
	We study a two-dimensional McKean-Vlasov stochastic differential equation, whose volatility coefficient depends on the conditional distribution of the second component with respect to the first component. We prove the strong existence and uniqueness of the solution, establishing the well-posedness of a two-factor local stochastic volatility (LSV) model calibrated to the market prices of European call options. In the spirit of \cite{jourdain2020existence}, we assume that the factor driving the volatility of the log-price takes finitely many values. Additionally, the propagation of chaos of the particle system is established, giving theoretical justification for the algorithm \cite{guyon2012being}.
\end{abstract}

\newtheorem{condition}[thm]{Condition}
\newtheorem{lemma}[thm]{Lemma}
\newcommand{\Var}[1]{\textrm{Var}\left(#1\right)}

\newcommand{\A}{\mathcal{A}}
\newcommand{\NN}{\mathbb{N}}
\newcommand{\sG}{\sqrt{\Gamma}}
\newcommand{\Sg}{\sqrt{\Gamma}^{-1}}
\newcommand{\Rns}{(\R^N_+)^*}

\newcommand{\Xb}{\overline{X}}
\newcommand{\sigmab}{\overline{\sigma}}
\newcommand{\Sigmab}{\overline{\Sigma}}

\newcommand{\Xt}{\widetilde{X}}
\newcommand{\Yt}{\widetilde{Y}}
\newcommand{\pt}{\widetilde{p}}
\newcommand{\sigmat}{\widetilde{\sigma}}
\newcommand{\Sigmat}{\widetilde{\Sigma}}
\newcommand{\ut}{\widetilde{u}}
\newcommand{\At}{\widetilde{A}}

\newcommand{\Ha}{H^{1+\alpha/2,2+\alpha}}

\newcommand{\inv}[1]{\frac{1}{#1}}
\newcommand{\Prob}[1]{\mathbb{P}\left(#1\right)}
\newcommand{\E}[1]{\mathbb{E}\left[#1\right]}
\newcommand{\R}{\mathbb{R}}
\newcommand{\Rd}{\mathbb{R}^d}
\newcommand{\f}[2]{\frac{#1}{#2}}
\renewcommand{\epsilon}{\varepsilon}
\newcommand{\VIX}{\textrm{VIX}}
\newcommand{\rI}{\rinterval}
\newcommand{\lI}{\linterval}
\intervalconfig{soft open fences}

\newcommand{\loc}{\textrm{loc}}
\newcommand{\lsv}{\textrm{lsv}}

\newcommand{\Xh}{\hat{X}}
\newcommand{\Yh}{\hat{Y}}
\newcommand{\ph}{\hat{p}}
\newcommand{\sigmah}{\hat{\sigma}}
\newcommand{\Sigmah}{\hat{\Sigma}}

\newcommand{\Io}[2]{\interval[open left]{#1}{#2}}
\newcommand{\IO}[2]{\interval[open right]{#1}{#2}}

\maketitle

\section{Introduction}

For the family of local stochastic volatility (LSV) models
\begin{equation}\label{LSV}
	\begin{cases}
		dS_t = \sigma_{\lsv}(t, S_t)a_tS_tdB_t,\ t\ge 0,      \\
		S_{t=0} = S_0,                                        \\
	\end{cases}
\end{equation}
to be calibrated to the market prices of European call options $(C(t,K))_{t> 0,K > 0}$ it is enough to take the leverage function $\sigma_{\lsv}$ as
\begin{equation}\label{du-pire}
	\sigma_{\lsv}(t, K) = \frac{\sigma_{\loc}(t, K)}{\sqrt{\E{a_t^2|S_t=K}}},\ t>0,K > 0,
\end{equation}
where $(S_t)_{t\ge 0}$ is the price process, $(B_t)_{t\ge 0}$ is a real standard Brownian motion, $S_0$ is a positive random variable independent of $(B_t)_{t\ge 0}$, $\sigma_{\loc}(t, K):=\sqrt{\frac{2\partial_tC(t,K)}{\partial_K^2C(t,K)}}, t> 0,K>0$ is the so-called Dupire volatility \cite{dupire1994pricing}  and $(a_t)_{t\ge 0}$ is any stochastic process. This is justified by Gy{\"o}ngy's theorem \cite{gyongy1986mimicking} that asserts that a stochastic process $(S_t)_{t\ge 0}$ solving \eqref{LSV} and \eqref{du-pire} has fixed-time marginal distributions given by the marginal distributions of the local volatility model $dS^{\loc}_t = \sigma_{\loc}(t, S^{\loc}_t)S^{\loc}_tdB_t$ and starting at $S_0$. In order to make the model tractable, one commonly assumes that $(a_t)_{t\ge 0}$ is given by $(\sqrt{f(Y_t)})_{t\ge 0}$ where $f:\R\to \R_+^*$ is a bounded smooth function and $(Y_t)_{t\ge 0}$ is another real It\^o-L\'evy process eventually correlated to $(B_t)_{t\ge 0}$.

Formally, the joint density $p(t, s, y)$ of $(S_t, Y_t),(t,s,y)\in \R_+\times\R_+^*\times\R$, solves a quasi-linear and non-local Fokker-Planck partial differential equation (PDE), whose coefficients depend upon the non-linear term
\begin{equation}
	\frac{\int p(\cdot, \cdot, z)dz}{\int f(z)p(\cdot, \cdot, z)dz}.
\end{equation}

While the applications are important in calibration to market implied volatility surfaces (\cite{henry2009calibration}, \cite[Chapter 11]{guyon2013nonlinear}, \cite{abergel2017nonparametric}, \cite{saporito2019calibration}) the existence and uniqueness of solutions to the stochastic differential equation (SDE) and the partial differential equation (PDE) problems is still an open problem. Only partial results have been obtained in particular cases.

Abergel and Tachet showed in \cite[Theorem 3.1]{abergel2010nonlinear}, the existence of a classical solution to the PDE problem in a bounded domain of $\R^2$ and with additional Dirichlet boundary condition - for short time, and for $\sup |f''|$ small enough. In \cite[Theorem 1.4]{lacker2020inverting}, the authors established the existence and uniqueness of a stationary solution to a similar SDE (a drift needs to be added to the dynamic of $(S_t)_{t\ge 0}$ to allow the possibility of a stationary measure). Jourdain and Zhou proved in \cite[Theorems 2.4, 2.5]{jourdain2020existence}, the existence of a weak solution to the SDE.  Under the assumption that $(Y_t)_{t\ge 0}$ is a jump process taking finitely many values and after writing the PDE problem as a system of parabolic equations in non-divergence form, the authors are able to write a variational formulation of the problem, for which existence of a solution can be proved - by the Galerkin's method -  provided that the range of $f$ is small enough. They are then able to prove the existence of a weak solution to the SDE. Uniqueness and propagation of chaos are out of reach in the approach of \cite{jourdain2020existence} because higher regularity of the weak solution is needed. Global regularity of weak solutions to parabolic systems is still an open problem.

The difficulty of the analysis of the McKean-Vlasov SDE and the related non-linear PDE stems from the singularity of the denominator $\int f(z)p(\cdot, \cdot, z)dz$. However, the well-posedness of McKean-Vlasov problems with coefficients depending smoothly on the density $p$ of the unknown process has been studied in various papers \cite{jourdain1998propagation}, \cite{oelschlager1985law}, \cite{kohatsu1997weak}, \cite{bossy2019wellposedness}.

The main contribution of our paper is the proof of well-posedness of a regularized version of SDE \eqref{LSV} - \eqref{du-pire} when $(a_t)_{t\ge 0}$ is time independent and given by $f(Y)$, where $Y$ is a fixed random variable independent of $(B_t)_{t\ge 0}$ and takes finitely many values in $\{1,\dots,N\}$, $N\ge 2$ and $f:\{1,\dots,N\}\to \R_+^*$. The regularization is chosen so that the calibration property is conserved. In other words, the fixed-time marginals of a solution are still given by the marginals of $S^{\loc}_t = S_0 + \int_0^t\sigma_{\loc}(s, S_s^{\loc})S_s^{\loc}dB_s$, $t\ge 0$.

Let $(X_t)_{t\ge 0} = (\log{S_t})_{t\ge 0}$ be the log price and set $\sigma(t, x) := \sigma_{\loc}(t, e^x), (t, x)\in \R_+\times\R$.
Herein, we assume that $Y$ is a random variable taking finitely many values $\{1,\cdots,N\}$, $N\ge 2$.
Let $\epsilon>0$ and consider the McKean-Vlasov SDE
\begin{equation}\label{lsv:SDE}
	\begin{cases}
		dX_t = -\inv{2}\frac{\epsilon + f(Y)p(t, X_t)}{\epsilon + \E{f(Y)|X_t}p(t, X_t)}\sigma(t, X_t)^2dt + \sqrt{\frac{\epsilon + f(Y)p(t, X_t)}{\epsilon + \E{f(Y)|X_t}p(t, X_t)}}\sigma(t, X_t)dB_t, \\
		\Prob{X_t\in dx} = p(t, x)dx,                                                                                                                                                                    \\
		X_{t=0} = X_0,
	\end{cases}
\end{equation}
where $X_0$ is a real random variable independent of $(B_t)_{t\ge 0}$.

We notice immediately
\begin{equation}
	\E{\left(\sqrt{\frac{\epsilon + f(Y)p(t, X_t)}{\epsilon + \E{f(Y)|X_t}p(t, X_t)}}\sigma(t, X_t)\right)^2\Big|X_t} = \sigma(t, X_t)^2,
\end{equation}
and consequently the model is calibrated, meaning $\exp(X_t)$ and $S_t^{\loc}$ have the same law for all $t\ge 0$.

In the general case where $(a_t)_{t\ge 0} = (f(Y_t))_{t\ge 0}$ and $Y$ is an It\^o process, the uniform ellipticity of the non-linear PDE - written in divergence form - solved by the joint density $p(t, x, y)$ of $(X_t, Y_t)$, $t\ge 0$, does not hold a priori. This property is a key element in establishing uniqueness in \cite[Equation 1.8]{jourdain1998propagation}. Therefore, we made the restrictive assumption that $Y$ is a time-independent discrete random variable and takes finitely many values. Assuming that the range of $f$ is small enough, we are able to derive a uniform ellipticity property for the PDE problem, in the spirit of what is done in \cite{jourdain2020existence}.

We assume in the whole paper that $\sigma$ is smooth, has bounded derivatives and there exist $0<\sigma_0 < \sigma_1$ such that $\sigma_0\le \sigma(t, x)\le\sigma_1$ for all $(t, x)\in\R_+\times\R$. Moreover, we assume that the measure $\Prob{X_0\in dx\cap Y=n}$ admits a density $P_n:\R\to\R_+$ - of total mass $\Prob{Y=n}$ - with bounded first and second derivatives and such that for some $\alpha\in (0, 1)$, the second derivative $P^{(2)}_n$  is $\alpha$-H\"older and
\begin{equation}
	\| P_n \|_{C^{2+\alpha}} := \sum_{0\le k\le 2}\|P^{(k)}_n\|_\infty+ \sup_{x, y\in \R,x\neq y}\frac{|P_n^{(k)}(x) - P_n^{(k)}(y)|}{|x-y|^\alpha} < \infty.
\end{equation}

Denote $f_{\max}=\max_{1\le n\le N} f(n)$, $f_{\min} = \min_{1\le n\le N}f(n)$ and
\begin{equation}
	\bar{f} = \inv{N}\sum_{n=1}^Nf(n).
\end{equation}

We introduce the following small range condition on $f$.
\begin{condition}\label{cond:f}
	\begin{equation}
		\begin{split}
			\inv{2}\left[N + 1 - \max_{1\le k\le N} \sqrt{\sum_{n=1,n\neq k}^Nf(n)\sum_{n=1,n\neq k}^N\inv{f(n)}}\right] \wedge 1 > &\inv{f_{\min}}\sqrt{\sum_{n=1}^N(f(n) - \bar{f})^2}\\
			&+ \frac{f_{\max} - f_{\min}}{f_{\min}}.\\
		\end{split}
	\end{equation}
\end{condition}

We are now ready to state our main result.
\begin{thm}\label{thm:existence-uniqueness}
	Let $T>0$. If Condition \ref{cond:f} is satisfied, $P$ belongs to $(C^{2+\alpha})^N,\alpha\in (0, 1)$ and the norm $\sum_{n=1}^N\|P_n\|_{C^{2+\alpha}}$ is small enough, then there exists a unique strong solution to SDE \eqref{lsv:SDE} on $[0,T]$. Moreover, $X$ admits a smooth density $p\in \left(C^{1+\alpha/4,2+\alpha/2}\right)^N$.
\end{thm}

We study as well the question of propagation of chaos for SDE \eqref{lsv:SDE}. Let $M\ge 1$, $(\delta_M)_{M\ge 1}\in (0,1)^{\NN^*}$ be a sequence converging to $0$ and $W_1 : \R\to\R_+^*$ be a bounded kernel function with a bounded derivative and satisfying $\int W_1(x)dx = 1$ and $\int xW_1(x)dx = 0$. Denote
\begin{equation}
	W_{\delta_M} := \inv{\delta_M}W_1\left(\frac{\cdot}{\delta_M}\right).
\end{equation}

Let $(B_t^i)_{t\ge 0,i\ge 1}$ be a collection of i.i.d standard Brownian motions and $(X_0^i, Y^i)_{i\ge 1}$ be a collection of i.i.d random variables of law $\Prob{Y^i=n\cap X_0^i\in dx}=P_n(x) dx,\ 1\le n\le N$.

For $M\ge 1$, we introduce the approximating system of $M$ particles, already considered in \cite{guyon2012being} (for $\epsilon = 0$)
\begin{equation}\label{SDE-particles}
	\begin{split}
		dX^{i,M}_t = &- \inv{2}\frac{\epsilon + f(Y^{i})\inv{M}\sum_{j=1}^M W_{\delta_M}(X^{i,M}_t-X_t^{j,M})}{\epsilon+ \inv{M}\sum_{j=1}^Mf(Y^j)W_{\delta_M}(X^{i,M}_t-X^{j,M})}\sigma(t, X^{i,M}_t)^2dt\\
		&+ \sqrt{\frac{\epsilon + f(Y^{i})\inv{M}\sum_{j=1}^M W_{\delta_M}(X^{i,M}_t-X_t^{j,M})}{\epsilon+ \inv{M}\sum_{j=1}^Mf(Y^j)W_{\delta_M}(X^{i,M}_t-X^{j,M})}}\sigma(t, X^{i,M}_t)dB^{i}_t,\\
	\end{split}
\end{equation}
initialized at $X^{i,M}_{t=0}= X_0^i$, $1\le i\le N$.

Assume that the conditions of Theorem \ref{thm:chaos} are satisfied. For each $1\le i\le M$, let $(\Xh^i)_{t\ge 0}$ be the particle starting at $X^{i,M}_0$ solving \eqref{lsv:SDE} and driven by the Brownian motion $(B^i_t)_{t\ge 0}$ and $Y^i$
\begin{equation}\label{SDEh}
	\begin{cases}
		\begin{aligned}
			d\Xh^{i}_t = & -\inv{2}\frac{\epsilon + f(Y^i)\sum_{n=1}^N\ph_n(t, \Xh^{i}_t)}{\epsilon + \sum_{n=1}^Nf(n)\ph_n(t, \Xh^{i}_t)}\sigma(t,\Xh_t^i)^2dt     \\
			             & + \sqrt{\frac{\epsilon + f(Y^i)\sum_{n=1}^N\ph_n(t, \Xh^{i}_t)}{\epsilon + \sum_{n=1}^Nf(n)\ph_n(t, \Xh^{i}_t)}}\sigma(t,\Xh_t^i)dB^i_t, \\
		\end{aligned} \\
		\Prob{\Xh^{i}_t\in dx\cap Y^i = n} = \ph_n(t, x)dx,                                                                                                                                           \\
		0\le t\le T.
	\end{cases}.
\end{equation}
The existence and uniqueness of the process $(X^{i,M}_t)_{0\le t\le T, 1\le i\le N}$ is ensured by the regularity of the drift and diffusion coefficients, given by Theorem \ref{thm:existence-uniqueness}.

Under Condition \ref{cond:f}, propagation of chaos holds.
\begin{thm}\label{thm:chaos}
	Let $T>0$, $(X^{i,M}_t)_{t\ge 0}$ and $(\Xh^i_t)_{t\ge 0}$ be given by \eqref{SDE-particles} and \eqref{SDEh} respectively. Assume that $(\delta_M)_{M\ge 1}$ converges slowly enough towards zero such that 
	\begin{equation}
		\lim_{M\to \infty} \inv{M}\frac{\exp{(CT\delta_M^{-4})}}{\delta_M^2} = 0.
	\end{equation}
	for all constants $C>0$.

	Then for all $1\le i\le N$
	\begin{equation}
		\lim_{M\to \infty}\E{\sup_{0\le t\le T}|X^i_t -\Xh^i_t|^2} = 0.
	\end{equation}
	Consequently, for all $k\ge 1$ the law of the process $(X^{1,M},X^{2,M},\dots, X^{k, M})$ converges weakly towards $\mu \otimes \mu \otimes \dots \otimes \mu$, where $\mu$ is the law of $\Xh^1$.
\end{thm}

The main application of this result is the justification of the particular method for calibrating the LSV model described in \cite{guyon2012being}, \cite[Section 11.6.1]{guyon2013nonlinear}.

The rest of the article is structured as follows. In Section \ref{sec:wellposedness} we introduce notation and prove Theorem \ref{thm:existence-uniqueness}. Existence is proved in Subsection \ref{subsec:existence} and uniqueness in Subsection \ref{subsec:uniqueness}. Section \ref{sec:chaos} is devoted to establishing propagation of chaos.


\section{Existence and uniqueness}\label{sec:wellposedness}

Throughout the rest of the paper $T>0$ is fixed, and

\begin{itemize}
	\item For any $d\ge 1$ and $X\in \R^d$, we denote $\|X\|_2= \sqrt{\sum_{i=1}^dX_i^2}$ and $\|X\|_\infty= \max_{1\le i\le d}|X_i|$.
	\item The scalar product is denoted $\langle\cdot,\cdot\rangle$
	      \begin{equation}
		      \langle X, Y\rangle = \sum_{i=1}^dX_iY_i,\ X\in\R^d,Y\in\R^d.
	      \end{equation}
	\item We denote by $L^p,p\ge 1$ the space of measurable real-valued functions $\phi$ for which $\|\phi\|_{L^p} = \left(\int |\phi(x)|^pdx\right)^{1/p}$ is finite.
	\item We denote by $C(0, T, L^2)$ the space of measurable functions $\phi : (t, x)\in [0,T]\times\R\to \R$ such that the mapping $t\in [0, T]\to \phi(t, \cdot)$ takes values in $(L^2,\|\|_{L^2})$ and is continuous.
	\item For $\alpha\in (0, 1)$ we denote by $C^{2+\alpha}$ the space of real functions $\phi$ on $\R$ for which $\|\phi\|_{C^{2+\alpha}} <\infty$.
	      We denote by $C^{(k+\alpha)/2,k+\alpha}$ the space of real functions $\phi$ on $[0,T]\times\R$ which are continuous together with their derivatives $\partial_t^r\partial_x^l\phi$, $2r+l\le k$ and admit a finite norm
	      \begin{equation}
		      \begin{split}
			      \|\phi\|_{C^{(k+\alpha)/2,k+\alpha}} &= \sum_{2r+l\le k}\sup_{(t, x)\in [0,T]\times\R}|\partial_t^r\partial_x^l\phi(t, x)|\\
			      & + \sum_{k-1\le 2r+l\le k} \sup_{x\in\R,t,s\in [0 ,T]}\frac{|\partial_t^r\partial_x^l\phi(t, x)-\partial^r_t\partial^l_x\phi(s, x)|}{|t-s|^{(k-2r-l+\alpha)/2}}\\
			      & + \sum_{2r+l= k} \sup_{t\in[0,T],x,y\in\R}\frac{|\partial_t^r\partial_x^l\phi(t, x)-\partial^r_t\partial^l_x\phi(t, y)|}{|x-y|^\alpha},
		      \end{split}
	      \end{equation}
	      where  $\partial_t$ and $\partial_x$ are the partial derivatives with respect to $t$ and $x$.
	\item $S^{++}$ is the set of symmetric and positive $N\times N$ real matrices.
\end{itemize}

For each $1\le n\le N$ and $t\ge 0$, denote by $p_n(t, x)$ the conditional density of $X_t$ given $Y=n$, multiplied by $\Prob{Y=n}$. In other words, $p_n(t, x)dx = \Prob{X_t\in dx\cap Y= n}, x\in\R, 1\le n\le N$. Set $P_n(x)dx = \Prob{X_0\in dx\cap Y=n}$ for $1\le n\le N$.

The PDE problem solved by $(p_n)_{1\le n\le N}$ can be formulated as a system of $N$ parabolic PDEs
\begin{equation}\label{SPDE}
	\begin{cases}
		\partial_t p = \inv{2}\partial_{xx}[\sigma^2B^\epsilon(p)p] + \inv{2}\partial_x\left[\sigma^2B^\epsilon(p)p\right],\ (t, x)\in [0, T]\times\R, \\
		p_n / \Prob{Y = n}\textrm{is a probability density,}                                                                                        \\
		p_n(0, x) = P_n(x),\ x\in \R, 1\le n\le N,                                                                                                     \\
	\end{cases}
\end{equation}
where $B^\epsilon(u), u\in (\R_+)^N$ is a diagonal matrix whose diagonal elements are
\begin{equation}
	B^\epsilon_{nn}(u) = B^\epsilon_n(u) := \frac{\epsilon + f(n)\sum_{k=1}^N u_k}{\epsilon + \sum_{k=1}^N f(k)u_k},\quad u\in \R_+^N, 1\le n\le N.
\end{equation}

The proof of Theorem \ref{thm:existence-uniqueness} is a direct application of Propositions \ref{prop:existence} and \ref{prop:uniqueness} below.

\subsection{Existence}\label{subsec:existence}
Under the condition that the norm $\sum_{n=1}^N\|P_n\|_{C^{2+\alpha}}$ is small enough, we prove the existence, in the classical sense, of a solution to Problem \eqref{SPDE} and the existence of a stochastic process $(X_t)_{t\ge 0}$ solving SDE \eqref{lsv:SDE}. The proof follows the approach of \cite[Proof of Proposition 2.2]{jourdain1998propagation} adapted to our specific McKean-Vlasov equation, whose coefficients depend on the marginal $\sum_{n=1}^Np_n$ and the quantity $\sum_{n=1}^Nf(n)p_n$.

\begin{prop}\label{prop:existence}
	There exists a constant $C$, depending on $T$, $f_{\max}$, $f_{\min}$ and $\epsilon$, such that if $\sum_{n=1}^N\| P_n\|_{C^{2+\alpha}} \le C$, then there exist a solution $(p_n)_{1\le n\le N}\in (C^{1+\alpha/2, 2+\alpha})^N$ to Problem \eqref{SPDE}, and a solution $(X_t)_{t\in [0, T]}$ to SDE \eqref{lsv:SDE}.
\end{prop}
\begin{proof}
	For $1\le n\le N$ and non-negative $u\in (C^{1+\alpha/2,2+\alpha})^N$, define the operator
	\begin{equation}
		\partial_tv - L_{n,u}v := \partial_tv - \inv{2}\partial_{xx}[\sigma^2B^\epsilon_n(u)v] - \inv{2}\partial_x[\sigma^2B^\epsilon_n(u)v].
	\end{equation}
	$L_{n,u}$ is a uniformly parabolic operator of second order with coefficients in $C^{\alpha/2,\alpha}$. According to \cite[Proposition 1.1]{jourdain1998propagation}(or \cite[Chapter IV, Theorem 5.1]{ladyzhenskaya1968linear}), there exists a unique $v_n\in C^{1+\alpha/2, 2+\alpha}$ such that $v_n / \Prob{Y_t=n}$ is a probability density and solves in $[0,T]\times\R$
	\begin{equation}\label{prop:pde-Ln}
		\begin{cases}
			\partial_tv_n = L_{n,u}v_n, \\
			v_n(0, x) = P_n(x),\ x\in\R.
		\end{cases}
	\end{equation}
	Moreover, there exists a constant $C'$ depending only on the regularity of the coefficients of $L_{n,u}$ (namely $f_{\max}, f_{\min}, \epsilon$ and  $\sum_{l=1}^N\|u_l\|_{1+\alpha/2,2+\alpha}$) such that
	\begin{equation}\label{eq:est-ussr}
		\|v_n\|_{C^{1+\alpha/2,2+\alpha}}\le C'\left(f_{\max}, f_{\min}, \epsilon, \sum_{l=1}^N\|u_l\|_{C^{1+\alpha/2,2+\alpha}}\right)\|P_n\|_{C^{2+\alpha}}.
	\end{equation}
	A useful fact about the constant $C'$ is that it is non-decreasing in its last argument.

	We construct by induction a sequence of solutions as follows. Start by $p^0_n(t, x) = P_n(x)$, $t\in [0, T], 1\le n\le N$ and $x\in\R$. For $m\ge 1$, let $p^m_n\in C^{1+\alpha/2,2+\alpha} $, for each $1\le n\le N$, be the solution of
	\begin{equation}
		\begin{cases}
			\partial_tp^m_n = L_{n, p^{m-1}}p^m_n,                    \\

			p^m_n / \Prob{Y = n}\textrm{is a probability density,} \\
			p^m_n(0, x) = P_n(x),\ x\in\R.
		\end{cases}
	\end{equation}

	Using estimate \eqref{eq:est-ussr}, we have for each $m\ge 1$
	\begin{equation}
		\sum_{n=1}^N\|p^m_n\|_{C^{1+\alpha/2,2+\alpha}}\le C'\left(f_{\max}, f_{\min}, \epsilon, \sum_{l=1}^N\|p^{m-1}_l\|_{C^{1+\alpha/2,2+\alpha}}\right)\sum_{n=1}^N\|P_n\|_{C^{2+\alpha}}.
	\end{equation}

	Under the assumption $\sum_{n=1}^N\|P_n\|_{2+\alpha}\le C:=1\wedge \inv{C'(f_{\max},f_{\min}, \epsilon, 1)}$, we have by induction that $\sum_{n=1}^N\|p^m_n\|_{C^{1+\alpha/2,2+\alpha}}\le 1$ for each $m\ge 1$. Indeed, the result holds for $m=0$. If $\sum_{n=1}^N\|p^{m-1}_n\|_{C^{1+\alpha/2,2+\alpha}}\le 1$ then
	\begin{equation}
		\sum_{n=1}^N\|p^m_n\|_{C^{1+\alpha/2,2+\alpha}}\le C'\left(f_{\max}, f_{\min}, \epsilon, 1\right)\sum_{n=1}^N\|P\|_{C^{2+\alpha}}\le 1.
	\end{equation}
	We proved that $(p^m)_{m\ge 0}$ is a bounded sequence of elements of $\left(C^{1+\alpha/2,2+\alpha}\right)^N$, and similarly to what is done in \cite[Proposition 2.2]{jourdain1998propagation},  we can extract a limit point $(p_n)_{1\le n\le N}\in (C^{1+\alpha/4,2+\alpha/2})^N$, such that $p_n / \Prob{Y=n}$ is a density of probability and $(p_n)_{1\le n\le N}$ solves Problem \eqref{SPDE}.

	Now, for each $1\le n\le N$, $(p_m)_{1\le m\le N}$ being regular, there exists a strong solution $(X_t, Y_t)_{t\in [0,T]}$ to the SDE
	\begin{equation}
		\begin{split}
			X_t = X_0 &- \inv{2}\int_0^t\frac{\epsilon + f(Y)\sum_{m=1}^N p_m(s, X^n_s)}{\epsilon +\sum_{m=1}^Nf(m)p_m(s, X^n_s)}\sigma(t,X_s)^2ds\\
			&+\int_0^t\sqrt{\frac{\epsilon + f(Y)\sum_{m=1}^N p_m(s, X^n_s)}{\epsilon + \sum_{m=1}^Nf(m)p_m(s, X^n_s)}}\sigma(t,X_s)dB_s.
		\end{split}
	\end{equation}
	Moreover, for each $1\le n\le N$ and $t\in [0,T]$, the law of $X_t1_{Y=n}$ admits a density $q_n(t,\cdot)$ solving $\partial_tq_n = L_{n, p}q_n$, with initial data $P_n$. The solution of this problem being unique and given by $p_n$, we conclude that $q_n = p_n$ for all $1\le n\le N$.
\end{proof}

\subsection{Uniqueness}\label{subsec:uniqueness}
Uniqueness is proved by first writing system \eqref{SPDE} in divergence form and then proving the uniform ellipticity of the differential operator.
For all $u\in C^1(\R\times \R_+)^N$, we can rewrite
\begin{equation}
	\partial_x[B^\epsilon(u)u] = A^\epsilon(u)\partial_xu,
\end{equation}
where $(A^\epsilon_{nk})_{1\le n,k\le N}: \R_+^{N}\to \R^{N\times N}$ is defined by
\begin{equation}
	A^\epsilon_{nn}(u):= B^\epsilon_n(u) + u_n\frac{f(n)\sum_{l=1}^N(f(l)-f(n))u_l}{(\epsilon + \sum_{l=1}^Nf(l)u_l)^2},\ 1\le n\le N,
\end{equation}
\begin{equation}
	A^\epsilon_{nk}(u):=u_n\frac{\epsilon(f(n)-f(k)) + f(n)\sum_{l=1}^N(f(l)-f(k))u_l}{(\epsilon + \sum_{l=1}^Nf(l)u_l)^2},\ 1\le n\neq k\le N.
\end{equation}


In these terms, PDEs \eqref{SPDE} can be rewritten
\begin{equation}
	\begin{cases}
		\partial_tp = \inv{2}\partial_x[\sigma^2A^\epsilon(p)\partial_xp] + \partial_x[\left(\sigma\partial_x\sigma +\inv{2}\sigma^2\right)B^\epsilon(p)p],\ (t,x)\in [0,T]\times\R, \\

		p_n / \Prob{Y = n}\textrm{is a probability density,}                                                                                                                      \\
		p_n(0,x) = P_n(x),\ x\in\R,1\le n\le N.
	\end{cases}
\end{equation}

We recall the following result of \cite[Proof of Proposition 2.3, Corollary B.3]{jourdain2020existence}.
\begin{prop}\label{prop:JZ}
	Let
	\begin{equation}
		\kappa_0 := \inv{2}\left[N + 1 - \max_{1\le k\le N} \sqrt{\sum_{n=1,n\neq k}^Nf(i)\sum_{n=1,n\neq k}^N\inv{f(n)}}\right].
	\end{equation}
	Then for all $\delta > 0$, $\eta>0$, $u\in (\R_+^*)^N$, $U\in \textbf{1}\R$ and $V\in \textbf{1}^\perp$, $\textbf{1} = (1)_{1\le n\le N}$,
	\begin{equation}
		\begin{split}
			\langle U + V, (J + \delta I)A^0(u)(U+V)\rangle &\ge \frac{N}{2}\left[N - \delta\left(1+\frac{f_{\max}}{f_{\min}}\right)\left(1 + \inv{2\eta}\right)\right]\|U\|_2^2\\
			&+ \delta\left(\kappa_0 - N^2\left(1 + \frac{f_{\max}}{f_{\min}}\right)\eta\right)\|V\|_2^2,
		\end{split}
	\end{equation}
	where $I = (1_{i=j})_{1\le i,j\le N}$ and $J = (1)_{1\le i,j\le N}$.
\end{prop}

We prove the following result.

\begin{prop}\label{prop:Gamma}
	Under Condition \ref{cond:f}, there exist a matrix $\Gamma\in S^{++}$ and $\kappa>0$ such that for all $u\in \R_+^N$ and $X\in \R^N$
	\begin{equation}
		\langle X, \Gamma A^\epsilon(u)X\rangle\ge \kappa\|X\|_2^2.
	\end{equation}
\end{prop}
\begin{proof}
	If $\sum_{n=1}^Nu_n = 0$ then $A^\epsilon(u) = I$ and the result holds. Assume that $\sum_{n=1}^Nu_n > 0$ and set
	\begin{equation}
		\rho:= \frac{\sum_{l=1}^Nf(l)u_l}{\epsilon + \sum_{l=1}^Nf(l)u_l} \in [0, 1].
	\end{equation}

	Rewrite
	\begin{equation}
		\begin{split}
			A^\epsilon_{nn}(u) &= B^\epsilon_n(u) + u_n\frac{f(n)\sum_{l=1}^N(f(l)-f(n))u_l}{(\epsilon + \sum_{l=1}^Nf(l)u_l)^2}\\
			&= 1-\rho + \rho B_n^0(u) + \rho^2 u_n\frac{f(n)\sum_{l=1}^N(f(l)-f(n))u_l}{(\sum_{l=1}^Nf(l)u_l)^2}\\
			&= (1-\rho)(1 + \rho B^0_n(u)) + \rho^2A^0_{nn}(u),
		\end{split}
	\end{equation}
	and
	\begin{equation}
		\begin{split}
			A^\epsilon_{nk}(u) &= u_n\frac{\epsilon(f(n)-f(k)) + f(n)\sum_{l=1}^N(f(l)-f(k))u_l}{(\epsilon + \sum_{l=1}^Nf(l)u_l)^2}\\
			&=  \rho(1-\rho)\frac{(f(n)-f(k))u_n}{\sum_{l=1}^Nf(l)u_l}+ \rho^2A^0_{nk}(u).
		\end{split}
	\end{equation}

	Define the matrix $D\in\R^{N\times N}$ by
	\begin{equation}
		D_{nk} := \frac{(f(n)-f(k))u_n}{\sum_{l=1}^Nf(l)u_l},\ 1\le n\neq k\le N,
	\end{equation}
	and
	\begin{equation}
		D_{nn} := -\sum_{m\neq n}D_{mn} = \frac{\sum_{m=1}^N(f(n)-f(m))u_m}{\sum_{l=1}^Nf(l)u_l} = B^0_n(u) - 1,\ 1\le n\le N.
	\end{equation}
	We see immediately that $\sum_{n=1}^ND_{nk} = 0$, $JD=0$, and that
	\begin{equation}
		\max_{1\le n,k\le N}|D_{nk}(u)|\le \frac{f_{\max}-f_{\min}}{f_{\min}}.
	\end{equation}

	In view of
	\begin{equation}
		(1-\rho)(1+\rho B^0_n) - D_{nn} = (1-\rho)(1+\rho B^0_n(u)) - \rho(1-\rho)(B^0_m(u)-1) = 1-\rho^2,
	\end{equation}
	we find that
	\begin{equation}
		A^\epsilon(u) = \rho^2 A^0(u) + (1-\rho^2)I + \rho(1-\rho)D.
	\end{equation}

	Set $M:= A^0(u)- I$. For all $1\le k\le N$, $\sum_{n=1}^NM_{nk}= 0$  according to \cite[Proof of Lemma 3.13]{jourdain2020existence} and therefore $JM = 0$. The previous equation can be written
	\begin{equation}
		A^\epsilon(u) = I + \rho^2M + \rho(1-\rho)D.
	\end{equation}

	Exploiting the idea of \cite[Proof of Proposition 2.3]{jourdain2020existence}, we are going to show that for $\delta\in (0, 1)$ small enough, there exists $\kappa>0$ such that for all $X\in\R^N$ and $u\in (\R^*_+)^N$
	\begin{equation}
		\langle X, (J+\delta I)A^\epsilon(u) X\rangle\ge \kappa\|X\|^2_2.
	\end{equation}

	Let $\delta\in (0, 1)$ and $X\in\R^N$. We decompose
	\begin{equation}
		(J+\delta I)A^\epsilon(u) = J+\delta I + \rho^2\delta M  + \rho(1-\rho)\delta D.
	\end{equation}

	If $\langle X, MX\rangle \le 0$, then using $\rho\in [0,1]$,
	\begin{equation}
		\begin{split}
			\langle X, 	(J+\delta I)A^\epsilon(u) X\rangle &= \langle X, (J+\delta I)X\rangle + \rho^2\delta\langle X, MX\rangle + \rho(1-\rho)\delta\langle X, DX\rangle\\
			&\ge \langle X, (J+\delta I)X\rangle + \delta\langle X, MX\rangle + \rho(1-\rho)\delta\langle X, DX\rangle\\
			&= \langle X, (J+\delta I)A^0X\rangle + \rho(1-\rho)\delta\langle X, DX\rangle.
		\end{split}
	\end{equation}

	If $\langle X, MX\rangle > 0$ then
	\begin{equation}
		\langle X, 	(J+\delta I)A^\epsilon(u) X\rangle\ge \langle X, (J+\delta I)X\rangle + \rho(1-\rho) \delta\langle X, DX\rangle.
	\end{equation}

	Write $X = U + V$ where $U\in \textbf{1}\R$ and $V\in \textbf{1}^\perp$. Then $DU = 0$ and
	\begin{equation}
		\langle X, DX\rangle =  \langle V, DV\rangle  + \langle U, DV\rangle.
	\end{equation}

	On the one hand,
	\begin{equation}
		\begin{split}
			\langle V, DV\rangle &= \sum_{n=1}^ND_{nn}V^2_n + \sum^N_{n, k=1}\frac{u_n(f(n)-f(k))}{\sum_{l=1}^Nf(l)u_l}V_nV_k\\
			&= \sum_{n=1}^ND_{nn}V^2_n + \sum^N_{n, k=1}\frac{u_n((f(n)-\bar{f})-(f(k)-\bar{f}))}{\sum_{l=1}^Nf(l)u_l}V_nV_k\\
			&= \sum_{n=1}^ND_{nn}V^2_n - \frac{\sum_{n,k=1}^N u_n (f(k)-\bar{f})V_nV_k}{\sum_{l=1}^Nf(l)u_l},
		\end{split}
	\end{equation}
	where we used that $\sum_{k=1}^NV_k=0$ between the second and third line.

	We estimate
	\begin{equation}
		\begin{split}
			\left|\sum_{n,k=1}^N u_n(f(k)-\bar{f})V_nV_k\right| &= \left|\sum_{n=1}^N u_nV_n\sum_{k=1}^N(f(k)-\bar{f})V_k\right|\\
			&\le \sqrt{\sum_{n=1}^Nu_n^2}\sqrt{\sum_{n=1}^N(f(n)-\bar{f})^2}\|V\|_2^2.
		\end{split}
	\end{equation}

	Using that $u_n> 0$ for $1\le n\le N$
	\begin{equation}
		\frac{\sqrt{\sum_{n=1}^Nu_n^2}}{{\sum_{l=1}^N f(l)u_l}}\le \inv{f_{\min}}\sqrt{\frac{\sum_{n=1}^Nu_n^2}{\left(\sum_{l=1}^N u_l\right)^2}}\le \inv{f_{\min}},
	\end{equation}
	we readily conclude that
	\begin{equation}\label{eq:V-MV}
		\rho(1-\rho)\delta \langle V, DV\rangle \ge -\delta\beta(f)\|V\|_2^2.
	\end{equation}
	where $\beta(f) := \left(\frac{f_{\max} -f_{\min}}{f_{\min}} +\inv{f_{\min}}\sqrt{\sum_{n=1}^N(f(n)-\bar{f})^2}\right)$,

	On the second hand, for all $\eta>0$
	\begin{equation}\label{eq:U-MV}
		\begin{split}
			\rho(1-\rho)\delta\langle U, DV\rangle &\ge - \delta \max_{1\le n,k\le N}|D_{nk}|N\|U\|_2\|V\|_2\\
			&\ge - \delta \beta(f)\sqrt{N}|\|U\|_2\|V\|_2\\
			&\ge - \delta\frac{\beta(f)N}{\eta}\|U\|^2_2 - \delta\beta(f) N^2\eta\|V\|_2^2.
		\end{split}
	\end{equation}

	Thus combining \eqref{eq:V-MV}, \eqref{eq:U-MV} and Proposition \ref{prop:JZ}, we infer in the case $\langle X, MX\rangle \le 0$
	\begin{equation}
		\begin{split}
			\langle X, (J + \delta I)A^\epsilon(u)X\rangle &\ge \frac{N}{2}\left[N - \delta\left[\left(1+\frac{f_{\max}}{f_{\min}}\right)\left(1 + \inv{2\eta}\right) +\frac{2\beta(f)}{\eta}\right]\right]\|U\|_2^2\\
			&+ \delta\left(\kappa_0 - \beta(f) - N^2\left(1 + \frac{f_{\max}}{f_{\min}} + \beta(f)\right)\eta\right)\|V\|_2^2.
		\end{split}
	\end{equation}

	Condition \ref{cond:f} yields $\kappa_0>\beta(f)$ and we can choose
	\begin{equation}
		\eta < \eta_{-} := \inv{2}\frac{(\kappa_0-\beta(f))}{N^2}\left(1+\frac{f_{\max}}{f_{\min}} + \beta(f)\right)^{-1},
	\end{equation}
	and
	\begin{equation}
		\delta < \delta_{-}(\delta):= \frac{N}{\left[\left(1+\frac{f_{\max}}{f_{\min}}\right)\left(1 + \inv{2\eta}\right) +\frac{2\beta(f)}{\eta}\right]}.
	\end{equation}

	We check that with such choice, if $\langle X, MX\rangle\le 0$
	\begin{equation}
		\langle X, (J + \delta I)A^\epsilon(u)X\rangle \ge \kappa_{-}\|X\|_2^2,
	\end{equation}
	where
	\begin{equation}
		\kappa_{-}:= \min\left( \frac{N}{2}\left[N - \delta\left[\left(1+\frac{f_{\max}}{f_{\min}}\right)\left(1 + \inv{2\eta}\right) +\frac{2\beta(f)}{\eta}\right]\right] ,\frac{\delta}{2}(\kappa_0-\beta(f))\right) > 0.
	\end{equation}

	Now in the case where $\langle X, MX\rangle > 0$
	\begin{equation}
		\begin{split}
			\langle X, (J+\delta I)A^\epsilon(u)X\rangle&\ge N\left[1 - \beta(f)\frac{\delta}{\eta}\right]\|U\|_2^2\\
			&+ \delta\left[1-\beta(f) - \beta(f)N^2\eta\right] \|V\|_2^2.
		\end{split}
	\end{equation}

	Therefore, under Condition \ref{cond:f}, $1 > \beta(f)$, and we can choose
	\begin{equation}
		\eta < \eta_+ := \inv{2}\frac{1-\beta(f)}{\beta(f)N^2},
	\end{equation}
	and
	\begin{equation}
		\delta < \delta_+(\eta) := \eta\beta(f).
	\end{equation}

	Thus
	\begin{equation}
		\langle U + V, (J + \delta I)A^\epsilon(u)(U+V)\rangle \ge \kappa_{+}\|X\|_2^2,
	\end{equation}
	where
	\begin{equation}
		\kappa_+ :=\min\left(N\left[1 - \beta(f)\frac{\delta}{\eta}\right],
		\frac{\delta}{2}\left[1-\beta(f)\right]\right) > 0.
	\end{equation}

	To conclude, take $\eta < \min(\eta_-,\eta_+)$ and $\delta < \min(\delta_-(\eta),\delta_+(\eta))$, to obtain
	\begin{equation}
		\langle X, (J+\delta I)A^\epsilon(u)X\rangle \ge \kappa\|X\|_2^2,
	\end{equation}
	where $\kappa := \min(\kappa_-, \kappa_+) > 0$.
\end{proof}
We are now ready to prove uniqueness.

\begin{prop}\label{prop:uniqueness}
	Under Condition \ref{cond:f}, there exists at most one solution in $C^{1+\alpha/4, 2+\alpha/2}$ to Problem \eqref{SPDE}.
\end{prop}
\begin{proof}
	Let $p$ and $q$ be two solutions in $(C^{1+\alpha/4,2+\alpha/2})^N$. By integrating by parts, we have the estimate
	\begin{equation}
		\|\partial_x p_n\|_{L^2}^2 \le \|\partial_{xx}p_n\|_{L^\infty}\|p_n\|_{L^1},\ 1\le n\le N
	\end{equation}
	we see that $\partial_xp, \partial_xq \in C(0, T, L^2)^N$.

	Let $\Gamma\in S^{++}$ and $\kappa>0$ be given by Proposition \ref{prop:Gamma}, such that for all $u\in (\R_+)^N$
	\begin{equation}
		\langle X, \Gamma A^\epsilon(u)X\rangle\ge \kappa\|X\|^2_2.
	\end{equation}

	Set $p' = \sG p$ and $q' =\sG q$. $p'$ and $q'$ are in $C^{1+\alpha/4, 2+\alpha/2}$, and their spatial gradient in $C(0, T, L^2)$, and they solve respectively the system of PDEs
	\begin{equation}\label{SPDE:Gamma}
		\begin{cases}
			\partial_tp' = \inv{2}\partial_x[\sigma^2A'(p)\partial_xp'] + + \partial_x[(\sigma\partial_x\sigma+ \inv{2}\sigma^2) B'(p)p'],\ (t,x)\in[0,T]\times\R, \\
			p'_n(0, x) = \sG P_n(x),\ x\in\R,1\le n\le N,
		\end{cases}
	\end{equation}
	and
	\begin{equation}
		\begin{cases}
			\partial_tq' = \inv{2}\partial_x[\sigma^2A'(q)\partial_xq'] + \partial_x[(\sigma\partial_x\sigma + \inv{2}\sigma^2)B'(q)q'],\ (t,x)\in [0,T]\times\R, \\
			q'_n(0, x) = \sG P_n(x),\ x\in\R,1\le n\le N,
		\end{cases}
	\end{equation}
	where the operators $A'$ and $B'$ are respectively defined by $A'(u):=\sG A^\epsilon(u)\Sg $ and $B'(u) := \sG B^\epsilon(u)\Sg$, $u\in (\R_+)^N$. $A'$ satisfies the same coercivity property as $A$. Indeed, for all $X\in\R$ and  $u\in (\R_+)^N$
	\begin{equation}\label{eq:A'-coercivity}
		\langle \sG X, \sG A^\epsilon(u)\Sg\sG X\rangle = \langle X, \Gamma A^\epsilon(u)X\rangle\ge \kappa\|X\|_2^2\ge \tilde{\kappa}\|\sG X\|_2^2,
	\end{equation}
	where $\tilde{\kappa}:= \frac{\kappa}{N + \delta}$, $(N+\delta)^{-1}$ being the smallest eigenvalue of $\Gamma^{-1}$.

	Multiplying \eqref{SPDE:Gamma}  by $p'-q'$ and integrating  gives for all $t\in [0, T]$
	\begin{equation}
		\begin{split}
			\int \langle p'-q',\partial_tp'\rangle &= -\inv{2}\int \langle \partial_x(p'-q'), A'(p)\partial_xp'\rangle\\
			&- \int \left(\sigma\partial_x\sigma + \inv{2}\sigma^2\right)\langle \partial_x(p'-q'), B'(p)p'\rangle.
		\end{split}
	\end{equation}
	A similar equation holds for $q'$. Taking the difference with the previous equation and integrating over $[0,t]$ gives
	\begin{equation}
		\begin{split}
			\inv{2}\int\|p'-q'\|_2^2&= -\inv{2}\int_0^t\int \langle \partial_x(p'-q'), A'(p)\partial_xp' - A'(q)\partial_xq'\rangle dt\\
			&- \int_0^t\int \left(\sigma\partial_x\sigma + \inv{2}\sigma^2\right)\langle \partial_x(p'-q'), B'(p)p' - B'(q)q'\rangle dt.
		\end{split}
	\end{equation}

	Rewrite the integrand of the first integral of the right-hand side as
	\begin{equation}
		\begin{split}
			\langle \partial_x(p'-q'), A'(p)\partial_xp' - A'(q)\partial_xq'\rangle &=
			\langle \partial_x(p'-q'), A'(p)\partial_x(p'-q')\rangle\\
			&+\langle \partial_x(p'-q'), (A'(p)-A'(q))\partial_xq'\rangle\\
			&\ge -\tilde{\kappa}\|\partial_x(p'-q')\|_2^2\\
			&- C\|q\|_{(C^{1+\alpha/2, 2+\alpha})^N}\|\partial_x(p'-q')\|_2\|A'(p)-A'(q)\|_2.
		\end{split}
	\end{equation}

	Likewise for the second integral
	\begin{equation}
		\begin{split}
			\sigma\partial_x\sigma\langle \partial_x(p'-q'), B'(p)p' - B'(q)q'\rangle &=
			\sigma\partial_x\sigma\langle \partial_x(p'-q'), B'(p)(p'-q')\rangle \\
			&+ \sigma\partial_x\sigma \langle \partial_x(p'-q'), (B'(p)-B'(q))q'\rangle \\
			&\ge -C \|\partial_x(p'-q')\|_2\|p'-q'\|_2\\
			&- C\|q\|_{(C^{1+\alpha/2, 2+\alpha})^N}\|\partial_x(p'-q')\|_2\|B'(p)-B'(q)\|_2,
		\end{split}
	\end{equation}
	where we used the boundedness of $B'(p)$, $\sigma$ and $\partial_x\sigma$.

	We check easily the existence of a constant $C$ depending on $\Gamma$, $f_{\max}$, $f_{\min}$ and $\epsilon$ such that
	\begin{equation}
		\|A'(p) - A'(q)\|_2^2 + 	\|B'(p) - B'(q)\|_2^2\le C\|p'-q'\|_2^2.
	\end{equation}

	In the view of the upper-bounds
	\begin{equation}
		\|\partial_x(p'-q')\|_2\|A'(p)-A'(q)\|_2\le \frac{\kappa}{4}\|\partial_x(p'-q')\|_2^2 + \frac{4}{\kappa}\|A'(p)-A'(q)\|_2^2,
	\end{equation}
	\begin{equation}
		\|\partial_x(p'-q')\|_2\|p'-q'\|_2\le \frac{\kappa}{4}\|\partial_x(p'-q')\|_2^2 + \frac{4}{\kappa}\|p'-q'\|_2^2,
	\end{equation}
	and
	\begin{equation}
		\|\partial_x(p'-q')\|_2\|B'(p)-B'(q)\|_2\le \frac{\kappa}{4}\|\partial_x(p'-q')\|_2^2 + \frac{4}{\kappa}\|B'(p)-B'(q)\|^2_2,
	\end{equation}
	we conclude that for all $t\in [0, T]$ and some constant $C>0$
	\begin{equation}
		\int \|p'-q'\|^2_2\le C\|q\|_{(C^{1+\alpha/4,2+\alpha/2})^N}^2\int_0^t\int \|p'-q'\|^2_2dt.
	\end{equation}

	Gronwall's inequality yields $p'=q'$ a.e. on $[0,T]\times\R$ and by continuity of $p$ and $q$, we have immediately that $p=q$.

\end{proof}

\section{Propagation of chaos}\label{sec:chaos}
In this section, we prove the propagation of chaos and introduce an intermediate mollified version of SDE \eqref{lsv:SDE}. For each $M\ge 1$, let $(\Xt^{i,M}_t)_{i\ge 1, t\ge 0}$ be a particle system such that the i-th particle starts at $X^i_0$ and evolves according to the dynamics
\begin{equation}\label{SDEt}
	\begin{cases}
		\begin{aligned}
			d\Xt^{i,M}_t = & - \inv{2}\frac{\epsilon + f(Y^i)\sum_{n=1}^NW_{\delta_M}*\pt_n(t, \Xt^{i,M}_t)}{\epsilon + \sum_{n=1}^Nf(n)W_{\delta_M}*\pt_n(t, \Xt^{i,M}_t)}\sigma(t,\Xt^{i,M}_t)^2dt \\ &+  \sqrt{\frac{\epsilon + f(Y^i)\sum_{n=1}^NW_{\delta_M}*\pt_n(t, \Xt^{i,M}_t)}{\epsilon + \sum_{n=1}^Nf(n)W_{\delta_M}*\pt_n(t, \Xt^{i,M}_t)}}\sigma(t,\Xt^{i,M}_t)dB^i_t,
		\end{aligned} \\
		\Prob{\Xt^{i,M}_t\in dx\cap Y^i = n} = \pt_n(t, x)dx,
	\end{cases},
\end{equation}
where we denote
\begin{equation}
	W_{\delta_M}*\phi(x) = \int W_{\delta_M}(x-y)\phi(y)dy,\quad \phi\in L^2(\R),\ x\in\R.
\end{equation}


We introduce the PDE system associated to SDE \eqref{SDEt}
\begin{equation}\label{PDEt}
	\begin{cases}
		\partial_t \pt = \inv{2}\partial_{xx}\left[\sigma^2B^\epsilon(W*\pt)\pt\right] + \inv{2}\partial_x\left[\sigma^2B^\epsilon(W*\pt)\pt\right],\ (t, x)\in [0,T]\times\R, \\
		\tilde{p}_n (0,x) = P_n(x),\ x\in \R,                                                                                                                                  \\
		1\le n\le N.
	\end{cases}.
\end{equation}


The existence of $(\Xt^{i,M}_t)_{1\le i\le M,t\ge 0}$ is ensured by the following proposition.

\begin{prop}\label{prop:estimate-pt-ph}
	If $(P_n)_{1\le n\le N}$ satisfies the assumption of Theorem \ref{thm:existence-uniqueness}, then there exists a strong solution $(\Xt^{i,M}_t, \pt(t,\cdot))_{t\ge 0}$ to SDE \eqref{SDEh}. Moreover, $(\pt_n)_{1\le n\le N}$ belongs to the space $(C^{1+\alpha/4,2+\alpha/2})^N \cap C(0, T, L^2)^N$ and there exists a constant $C>0$ independent of $M$ and $\delta_M$ such that
	\begin{equation}\label{prop:estimate-pt}
		\sum_{n=1}^N\|\tilde{p}_n\|_{C^{1+\alpha/4,2+\alpha/2}}\le C,
	\end{equation}
	and
	\begin{equation}\label{prop:eq-estimate-pt-ph}
		\sup_{0\le t\le T}\int \|\pt - \ph\|_2^2\le C\delta_M,\ \forall t\ge 0.
	\end{equation}
\end{prop}
\begin{proof}
	The existence of $(\tilde{p}_n)_{1\le n\le N}$ and $(\tilde{X}^{i,M}_t)_{t\ge 0}, 1\le i\le M$, can easily be established by noticing that the mapping $\phi\in C^{1+\alpha/4,2+\alpha/2} \mapsto W_{\delta_M}* u\in  C^{1+\alpha/4,2+\alpha/2}$ is non-expansive for the norm $\|\cdot\|_{C^{1+\alpha/4,2+\alpha/2}}$ and by using the techniques of the proof of Proposition \ref{prop:existence} (see as well \cite[Proposition 2.2]{jourdain1998propagation}). The non-expansivity of the mentioned mapping makes the estimate \eqref{eq:est-ussr} still valid and independent of $\delta_M$, and therefore \eqref{prop:estimate-pt} is justified.

	In order to establish \eqref{prop:eq-estimate-pt-ph}, we shall show that the proof of  \cite[Lemma 2.6]{jourdain1998propagation} carries to parabolic systems. Mimicking the computation of \cite[Equation (2.8)]{jourdain1998propagation}, we can write that the difference $\ph -\pt$ solves the system
	\begin{equation}\label{eq:ph-pt}
		\partial_t(\ph-\pt) = \inv{2}\partial_x[A^\epsilon(\ph)\partial_x(\ph - \pt)] + a_1\partial_x(\ph - \pt) + a_2(\ph-\pt) + \phi.
	\end{equation}
	where the matrix valued functions $a_1(t, x)$ and $a_2(t, x)$ are bounded, depend on $\hat{p}$ and $\tilde{p}$, and the vector valued function $\phi(t, x)$ depends on $\pt - W * \pt$ and its first and second derivatives. Moreover, $\phi$ satisfies
	\begin{equation}
		\sum_{n=1}^N\sup_{x\in\R}|\phi_n(x)| \le C\delta_M,
	\end{equation}
	for some constant $C>0$ independent of $\delta_M$.

	Next, we observe that under Condition \ref{cond:f}, there exist $\Gamma\in S^{++}$ and $\kappa>0$ such that for all $X\in\R^N$
	\begin{equation}
		\langle X,\Gamma A^\epsilon(\ph)X\rangle \ge \kappa\|X\|_2^2.
	\end{equation}

	Like in the proof of Proposition \ref{prop:uniqueness}, set $\ph' = \sG\ph$, $\pt'=\sG\pt$, $A'(\ph)=\sG A^\epsilon(\ph)\Sg$, $a_1'=\sG a_1\Sg$, $a_2'=\sG a_2\Sg$ and $\phi'=\sG\phi$. With \eqref{eq:A'-coercivity}, there exists $\kappa'>0$ such that for all $X\in\R^N$
	\begin{equation}
		\langle X, A'X\rangle \ge \kappa'\|X\|^2_2.
	\end{equation}
	Moreover
	\begin{equation}
		\begin{split}
			\inv{2}\int \|\ph' - \pt'\|_2^2 = &- \inv{2}\int_0^t\int \langle \partial_x(\ph'-\pt'),A'(\ph)\partial_x(\ph'-\pt')\rangle dt\\
			&+ \int_0^t\int \langle \ph'-\pt', a_1'\partial_x(\ph'-\pt')\rangle dt
			+ \int_0^t\int \langle \ph'-\pt', a_2(\ph'-\pt')\rangle dt\\
			&+ \int_0^t\int \langle\ph'-\pt',\phi'\rangle dt.
		\end{split}
	\end{equation}

	Using the boundedness of $a_1'$, $a_2'$ and standard techniques, it is easy to see that for some appropriate constant $C > 0$ and for all $t\in [0,T]$
	\begin{equation}
		\begin{split}
		  &\int \langle \ph'-\pt', a_1'\partial_x(\ph'-\pt')\rangle + \int \langle \ph'-\pt', a_2(\ph'-\pt')\rangle\\
		  &\le \frac{\kappa'}{2}\int \|\partial_x(\hat{p}'-\tilde{p}')\|^2_2 + C\int \|\hat{p}'-\tilde{p}'\|^2_2.
		\end{split}
	\end{equation}

	In the view of the uniform coercivity of $A'$, and $\phi'$, and the estimate
	\begin{equation}
		\int \langle\ph'-\pt',\phi'\rangle\le C\delta_M \sum_{n=1}^N\int (\ph_n + \pt_n) = C\delta_M,
	\end{equation}
	we conclude that there exists some constant $C>0$ such that for all $t\in [0,T]$
	\begin{equation}
		\inv{2}\int \|\ph' - \pt'\|_2^2 \le C\delta_M  + C\int_0^t\inv{2}\int \|\ph' - \pt'\|_2^2dt.
	\end{equation}

	The result is proved by applying Gronwall's inequality and using the obvious inequality  $\|\ph - \pt\|_2\le \|\ph' - \pt'\|_2$, for some constant $C>0$ depending only on $\Gamma$ and independent of $\delta_M$.
\end{proof}

Let $1\le i\le M$.
To simplify the notation, we will drop the subscripts and superscripts in $i$ and $M$, and we denote $\delta=\delta_M$, $W=W_{\delta_M}$, $X = X^i$, $\Xt = \Xt^i$, $\Xh = \Xh^i$, and for $j\neq i$, $X^{j} = X^{j,M}$ and $\Xt^{j} = \Xt^{j,M}$.
Denote $W_{\delta_M}*X_t(x) = \inv{M}\sum_{j=1}^MW_{\delta_M}(x-X^{j}_t)$ and $W_{\delta_M}^f*X_t(x) = \inv{M}\sum_{j=1}^Mf(Y^j)W_{\delta_M}(x-X^j_t)$. Under this notation, the volatility coefficients of $X$, $\Xt$ and $\Xh$ are respectively given by
\begin{equation}
	\Sigma_t := f(Y)\sqrt{\frac{\epsilon +  W * X_t(X_t)}{\epsilon + W^f*X_t(X_t)}}\sigma(t, X_t),
\end{equation}
\begin{equation}
	\Sigmat_t := \sqrt{\frac{\epsilon + f(Y)\sum_{n=1}^NW*\pt_n(t, \Xt_t)}{\epsilon+\sum_{n=1}^Nf(n)W*\pt_n(t, \Xt_t)}}\sigma(t, \Xt_t),
\end{equation}
and
\begin{equation}
	\Sigmah_t := \sqrt{\frac{\epsilon + f(Y)\sum_{n=1}^N\ph_n(t, \Xh_t)}{\epsilon+\sum_{n=1}^Nf(n)\ph_n(t, \Xt_t)}}\sigma(t, \Xh_t).
\end{equation}

Notice immediately that
\begin{equation}
	\sigma_0^2\frac{\inf f}{\sup f}\le \Sigma^2,\Sigmat^2,\Sigmah^2\le\frac{\sup f}{\inf f}\sigma_1^2.
\end{equation}

In the rest of the section, we denote, for notational simplicity, by $C$ any constant depending on $N$, $\epsilon$, $f$, $\sigma_0$, $\sigma_1$, and $|W_1|_\infty$.

The proof of Theorem \ref{thm:chaos} requires the Propositions \ref{prop:X-Xt} and \ref{prop:Xt-Xh} below.

\begin{prop}\label{prop:X-Xt}
	Let $1\le i\le M$. Then for all $T>0$.
	\begin{equation}\label{eq:X-Xt}
		\E{\sup_{0\le t\le T}(X^{i,M}_t-\Xt^{i,M}_t)^2}\le\frac{C}{M}\frac{\exp{(CT\delta_M^{-4})}}{\delta_M^2}.
	\end{equation}
\end{prop}

\begin{proof}
	According to the Burkholder–Davis–Gundy inequality, for any $s\in [0,T]$
	\begin{equation}\label{eq:BDG}
		\begin{split}
			\E{\sup_{0\le t\le s} (X_t -\Xt_t)^2}&\le C\E{\langle X - \Xt\rangle_s}\\
			&=C\int_0^s\E{(\Sigma_t-\Sigmat_t)^2}dt.
		\end{split}
	\end{equation}

	For $x\in\R$ and $(U^j)_{1\le j\le M}\in\R^M$, we abbreviate  $\inv{M}\sum_{j=1}^M W(x - U^j)$ by $W*U(x)$ and $\inv{M}\sum_{j=1}^Mf(Y^j)W(x-U^j)$ by $W^{f}*U(x)$.

	The boundedness $\Sigma_t$, $\Sigmat_t$, $\sigma$ and $\partial_x\sigma$ gives for all $t\in [0, T]$
	\begin{equation}\label{eq:sigma-sigmat}
		\begin{split}
			|\Sigma_t - \Sigmat_t|&\le C|\Sigma_t^2-\Sigmat_t^2|\\
			&\le C\left|\frac{\epsilon + W^{f}*X_t(X_t)}{\epsilon + W*X_t(X_t)} - \frac{\epsilon + \sum_{n=1}^N f(n)W*\pt_n(t, \Xt_t)}{\epsilon + \sum_{n=1}^N W*\pt_n(t, \Xt_t)}\right|\\
			&+ C|\sigma(t, X_t)-\sigma(t, \Xt_t)|\\
			&\le C\left|\frac{\epsilon + W^{f}*X(X_t)}{\epsilon + W*X_t(X_t)} - \frac{\epsilon + W^{f}*X_t(X_t)}{\epsilon + \sum_{n=1}^NW*\pt_n(t, \Xt_t)}\right|\\
			&+ C\left|\frac{\epsilon + W^{f}*X(X_t)}{\epsilon + \sum_{n=1}^NW*\pt_n(t, \Xt_t)} - \frac{\epsilon + \sum_{n=1}^Nf(n)W*\pt_n(t,\Xt_t)}{\epsilon + \sum_{n=1}^NW*\pt_n(t, \Xt_t)}\right|\\
			&+ C|X_t-\Xt_t|\\
			&\le C\left| W*X_t(X_t) - \sum_{n=1}^N W*\pt_n(t, \Xt_t)\right|\\
			&+
			C\left| W^{f}*X_t(X_t) - \sum_{n=1}^Nf(n)W*\pt_n(t,\Xt_t)\right| + C|X_t-\Xt_t|.
		\end{split}
	\end{equation}

	Let $t\in [0,T]$. After estimating the term $\E{\left| W^{f}*X_t(X_t) - \sum_{n=1}^Nf(n)W*\pt_n(t,\Xt_t)\right|^2}$, we will apply the estimate with $f=1$ to treat the term $\E{\left| W*X_t(X_t) - \sum_{n=1}^N W*\pt_n(t, \Xt_t)\right|^2}$.

	Upper-bound
	\begin{equation}
		\begin{split}
			\left|W^{f}*X_t(X_t) - \sum_{n=1}^Nf(n)W*p_n(t, \Xt_t)\right|
			&\le (\textrm{I}) + (\textrm{II}) + (\textrm{III}),
		\end{split}
	\end{equation}
	where each term $(\textrm{I}), (\textrm{II})$ and $(\textrm{III})$ is respectively defined by
	\begin{equation}
		(\textrm{I}) :=  \left|W^{f}*X_t(X_t) - W^{f}*X_t(\Xt_t)\right|,
	\end{equation}
	\begin{equation}
		(\textrm{II}) := \left| W^{f}*X_t(\Xt_t)- W^{f}*\Xt_t(\Xt_t)\right|,
	\end{equation}
	and
	\begin{equation}
		(\textrm{III}) := \left|W^{f}*\Xt_t(\Xt_t) - \sum_{n=1}^Nf(n)W*\pt_n(t, \Xt_t)\right|.
	\end{equation}

	By using $\sup_{x\in\R}|W'(x)| \le \frac{C}{\delta^2}$, we find
	\begin{equation}\label{eq:I}
		\begin{split}
			\E{(\textrm{I})^2} &= \E{\left(\inv{M}\sum_{j=1}^M\E{f(Y^j)(W(X_t-X^j_t) - W(\Xt_t-X^j_t))}\right)^2}\\
			&\le \frac{C}{\delta^4}\E{|X_t-\Xt_t|^2}.
		\end{split}
	\end{equation}

	Likewise
	\begin{equation}\label{eq:II}
		\begin{split}
			\E{(\textrm{II})^2} &= \E{\left(\inv{M}\sum_{j=1}^Mf(Y^j)(W(\Xt_t-X^j_t) - W(\Xt_t-\Xt^j_t))\right)^2}\\
			&\le \frac{C}{\delta^4M}\E{\left(\sum_{j=1}^M|X^j_t-\Xt^j_t|^2\right)}\\
			&\le \frac{C}{\delta^4}\E{|X_t-\Xt_t|^2},
		\end{split}
	\end{equation}
	where we used that by symmetry $\E{|X^j_t-\Xt^j_t|^2} = \E{|X_t-\Xt_t|^2}$ for all $j=1,\dots,M$.

	By independence of the $(\Xt^j_t)_{1\le j\le M}$ and using that
	\begin{equation}
		\sum_{n=1}^Nf(n)W*\pt_n(t, x) = \E{f(Y^j)W(x-\Xt^j_t)},\ 1\le j\neq i\le N,
	\end{equation}
	we get
	\begin{equation}
		\begin{array}{c}
			\E{(\textrm{III})^2} =  \displaystyle\int \E{\left(\inv{M}\sum_{j\neq i}f(Y^j)W(x-\Xt^j_t)  - \E{f(Y^j)W(x-\Xt^j_t) }\right)^2} \\
			\times\Prob{\Xt_t\in dx}.
		\end{array}
	\end{equation}

	Therefore
	\begin{equation}\label{eq:III-1}
		\begin{split}
			\E{(\textrm{III})^2} &\le C\int \E{\left(\inv{M}\sum_{j\neq i}\left(f(Y^j)W(x-\Xt^j_t)  - \E{f(Y^j)W(x-\Xt^j_t) }\right)\right)^2}\Prob{\Xt_t\in dx}\\
			&+ \frac{C}{M^2}\int \E{f(Y)W(x-\Xt_t) }^2\Prob{\Xt_t\in dx}.
		\end{split}
	\end{equation}

	Making use of the bound $\sup_{x\in \R}|W(x)|\le \frac{C}{\delta}$, we see that the second integral in the right-hand side of \eqref{eq:III-1} is dominated by
	\begin{equation}
		\begin{split}
			\int \E{f(Y)W(x-\Xt_t) }^2\Prob{\Xt_t\in dx} &\le C\int W(x-y)^2\Prob{\Xt_t\in dx}\Prob{\Xt_t\in dy}\\
			&\le \frac{C}{\delta^2}.
		\end{split}
	\end{equation}

	In estimating the first integral of \eqref{eq:III-1}, we utilize the independence of the $(\tilde{X}^j_t)_{1 \le j\le N}$, for all $x\in\R$, which gives
	\begin{equation}\label{eq:III-last}
		\begin{split}
			& \E{\left(\inv{M}\sum_{j\neq i}\left(f(Y^j)W(x-\Xt^j_t) - \E{f(Y^j)W(x-\Xt^j_t)} \right)\right)^2}\\
			&= \Var{\inv{M}\sum_{j\neq i}f(Y^j)W(x-\Xt^j_t)}\\
			&= \frac{M-1}{M^2}\Var{f(Y)W(x-\Xt_t)}\\
			&\le \frac{C}{M\delta^2}.
		\end{split}
	\end{equation}

	Putting together \eqref{eq:I}, \eqref{eq:II} and \eqref{eq:III-1}-\eqref{eq:III-last}, we infer
	\begin{equation}
		\E{\left|W^{f}*X_t(X_t) - \sum_{n=1}^Nf(n)W*\pt_n(t,\Xt_t)\right|^2}\le \frac{C}{M\delta^2} + \frac{C}{\delta^4}\E{|X_t-\Xt_t|^2}.
	\end{equation}

	Taking $f = 1$ leads to the same estimate for $\E{\left|W*X_t(X_t) - \sum_{n=1}^NW*\pt_n(t,\Xt_t)\right|^2}$.

	If we insert this in \eqref{eq:sigma-sigmat} and use \eqref{eq:BDG}, we obtain
	\begin{equation}
		\E{\sup_{0\le s\le t}(X_s-\Xt_s)^2} \le \inv{M}\frac{C}{\delta^2} + \frac{C}{\delta^4}\int_0^t\E{\sup_{0\le s\le r}(X_s-\Xt_s)^2}dr.
	\end{equation}

	Applying Gronwall's inequality proves \eqref{eq:X-Xt} and concludes the proof of the proposition.
\end{proof}

\begin{prop}\label{prop:Xt-Xh}
	Let $1\le i\le M$. Then for some constant $C>0$ independent of $\delta$
	\begin{equation}
		\E{\sup_{0\le t\le T}|\Xt^{i,M}_t - \Xh^i_t|^2} \le C\delta_M.
	\end{equation}
\end{prop}

\begin{proof}
	In the same manner of \eqref{eq:sigma-sigmat}, for some constant $C>0$
	\begin{equation}
		|\Sigmat_t - \Sigmah_t|
		\le \frac{C}{\delta}\sum_{n=1}^N\left|W*\pt_n(t,\Xt_t) - \ph_n(t,\Xh_t)\right| + C|\Xt_t-\Xh_t|.
	\end{equation}

	For all $t\in [0, T]$
	\begin{equation}\label{eq:Xt-Xh-1}
		\begin{split}
			\left|W*\pt_n(t,\Xt_t) - \ph_n(t,\Xh_t)\right|&\le
			\left|W*\ph_n(t,\Xh_t) - \ph_n(t,\Xh_t)\right|\\
			&+ \left|W*\ph_n(t,\Xt_t) - W*\ph_n(t,\Xh_t)\right|\\
			&+ \left|W*\pt_n(t,\Xt_t) - W*\ph_n(t,\Xt_t)\right|.
		\end{split}
	\end{equation}

	Next, using the boundedness of the gradient of $\ph_n$, we estimate the  first two terms of the left-hand side of \eqref{eq:Xt-Xh-1} as follows
	\begin{equation}
		\begin{split}
			\left|W*\ph_n(t,\Xh_t) - \ph_n(t,\Xh_t)\right| &\le \sup_{x\in \R}\int W_1(y)|\ph_n(t, x - \delta y) - \ph_n(t,x)|dy\\
			&\le \delta\|\ph_n\|_{C^{1+\alpha/2,2+\alpha}}\int |y|W_1(y)dy,
		\end{split}
	\end{equation}
	and
	\begin{equation}
		\left|W*\ph_n(t,\Xt_t) - W*\ph_n(t,\Xh_t)\right|\le \|\ph_n\|_{C^{1+\alpha/2,2+\alpha}}|\Xt_t-\Xh_t|.
	\end{equation}

	Therefore, there exists a constant $C>0$ independent of $\delta$ such that
	\begin{equation}
		\left|W*\ph_n(t,\Xh_t) - \ph_n(t,\Xh_t)\right| + \left|W*\ph_n(t,\Xt_t) - W*\ph_n(t,\Xh_t)\right|\le C(\delta + |\Xt_t - \Xh_t|).
	\end{equation}

	For the last term of \eqref{eq:Xt-Xh-1}, write
	\begin{equation}
		\begin{split}
			\E{\left|W*\pt_n(t,\Xt_t) - W*\ph_n(t,\Xt_t)\right|^2} &= \sum_{m=1}^N\int |W*(\pt_n - \ph_n)|^2p_m\\
			&\le \|p\|_{(C^{1+\alpha/4,2+\alpha/2})^N}\int |W*(\pt_n - \ph_n)|^2\\
			&\le  \|p\|_{(C^{1+\alpha/4,2+\alpha/2})^N}\int |\pt_n - \ph_n|^2\\
			&\le  \|p\|_{(C^{1+\alpha/4,2+\alpha/2})^N}\delta,
		\end{split}
	\end{equation}
	where we used Proposition \ref{prop:estimate-pt-ph}.

	Combining the previous estimate and \eqref{eq:Xt-Xh-1} we deduce the existence of $C>0$ independent of $\delta$ such that
	\begin{equation}
		\E{|\sigmat_t - \sigmah_t|^2} \le C\left(\delta^2 + \E{|\Xt_t-\Xh_t|^2}\right).
	\end{equation}

	Burkholder–Davis–Gundy inequality and Gronwall's inequality yield
	\begin{equation}
		\E{\sup_{0\le t\le T}|\Xt_t - \Xh_t|^2} \le C\delta.
	\end{equation}
\end{proof}

We are now ready for the proof of Theorem \ref{thm:chaos}.
\begin{proof}[Proof of Theorem \ref{thm:chaos}]
	According to Propositions \ref{prop:X-Xt} and \ref{prop:Xt-Xh} above, there exists a constant $C>0$ such that for all $1\le i\le N$
	\begin{equation}
		\E{\sup_{0\le t\le T}\left(X^{i,M}_t - \Xh^{i, M}_t\right)^2} \le \frac{C}{M}\frac{\exp{(CT\delta_M^{-4})}}{\delta_M^2} + C\delta_M.
	\end{equation}

	It is enough to choose $(\delta_M)_{M\ge 1}$ such that $\lim_{M\to \infty}\delta_M = 0$ and
	\begin{equation}
		\lim_{M\to \infty} \inv{M}\frac{\exp{(CT\delta_M^{-4})}}{\delta_M^2} = 0.
	\end{equation}
\end{proof}

\bibliographystyle{plain}
\bibliography{bib}

\end{document}